\documentclass[11pt]{article}

\usepackage[utf8]{inputenc}
\usepackage[T1]{fontenc}
\usepackage[english]{babel}
\usepackage{amsmath,amssymb,amsthm}
\usepackage{hyperref}
\usepackage{url}
\usepackage{tikz}
\usepackage{float}
\usepackage{array}
\usepackage{booktabs}
\usepackage{tabularx}
\usepackage{microtype}
\hypersetup{colorlinks=true,linkcolor=blue,citecolor=blue,urlcolor=blue,hypertexnames=false}

\providecommand{\keywords}[1]{\noindent\textbf{\textit{Keywords---}} #1}
\newcommand{\res}{\operatorname{res}}

\newcommand{\ceil}[1]{\left\lceil #1\right\rceil}
\newcommand{\floor}[1]{\left\lfloor #1\right\rfloor}

\newtheorem{theorem}{Theorem}[section]
\newtheorem{conjecture}[theorem]{Conjecture}
\newtheorem{corollary}[theorem]{Corollary}
\newtheorem{definition}[theorem]{Definition}
\newtheorem{lemma}[theorem]{Lemma}

\newtheorem{proposition}[theorem]{Proposition}
\newtheorem{remark}[theorem]{Remark}

\title{Annihilation, Independence, and Residue:\\
\large Sharp Matching Bounds for the Annihilation Gap\\ and a TxGraffiti Application}
\author{
        Ohr Kadrawi \\
        Department of Mathematics\\
        Ariel University\\
        Ariel 4070000, Israel\\
        orka@ariel.ac.il
        \and
        Vadim E. Levit \\
        Department of Mathematics\\
        Ariel University\\
        Ariel 4070000, Israel\\
        levitv@ariel.ac.il
}
\date{}

\begin{document}

\maketitle

\begin{abstract}
Let $G$ be a finite simple graph.  The annihilation number $a(G)$ is an efficiently computable upper bound on the independence number $\alpha(G)$.  We develop a sharp matching-number theory for the gap $a(G)-\alpha(G)$.  The strongest general theorem is the exact closed form
\[
        a(G)-\alpha(G)\leq 2\mu(G)+1-\ceil{\sqrt{6\mu(G)}}\qquad(\mu(G)\geq 1),
\]
and the bound is attained for every prescribed matching number.  We also prove sharp matching-dependent bounds for forests, bipartite graphs, and K\"onig--Egerv\'ary graphs, with equality constructions, equality certificates, and equality criteria.

Finally, we treat a TxGraffiti output as a machine-conjecture case study.  Using annihilating decompositions together with the classical Havel--Hakimi residue inequality $\res(G)\leq \alpha(G)$, we give an independent proof of the TxGraffiti annihilation-residue inequality
\[
        \alpha(G)\geq \frac{a(G)+\res(G)}{\Delta(G)}
\]
for every connected graph $G$ of order at least three, show that both hypotheses are necessary, and compare this proof with a recent Caro--Wei approach.  We also refine the Caro--Wei annihilation estimate by an explicit nonnegative slack term, identify its equality cases in degree-sequence form, and combine the refinement with our exact matching-number bound to obtain a combined computable bracket for the independence number and a Gupta--residue bound for the annihilation gap.
\end{abstract}



\section{Introduction}

Throughout the paper all graphs are finite, simple, and undirected.  Standard graph terminology follows West \cite{West2001}.  For a graph $G$, let $V(G)$ and $E(G)$ denote its vertex and edge sets, and put $n(G)=|V(G)|$ and $m(G)=|E(G)|$.  We write $d_G(v)$, or simply $d(v)$, for the degree of a vertex $v$, and $\Delta(G)$ for the maximum degree of $G$.

An independent set is a set of pairwise non-adjacent vertices.  The independence number $\alpha(G)$ is the maximum cardinality of an independent set.  A matching is a set of pairwise disjoint edges, and the matching number $\mu(G)$ is the maximum cardinality of a matching.  We write $\tau(G)=n(G)-\alpha(G)$ for the vertex cover number; this identity follows because the complement of an independent set is a vertex cover, and conversely.  A graph $G$ is called a K\"onig--Egerv\'ary graph if
\[
        \alpha(G)+\mu(G)=n(G).
\]
This terminology and its matching-cover characterizations go back to the work of Deming, Gavril, and Sterboul \cite{Deming1979,Gavril1977,Sterboul1979}.  The classical K\"onig--Egerv\'ary theorem for bipartite graphs states that the maximum matching size equals the minimum vertex-cover size; we cite both original papers, by K\H{o}nig and by Egerv\'ary, published in the same 1931 volume of Matematikai \'es Fizikai Lapok \cite{Egervary1931,Koenig1931}.  Consequently, every bipartite graph is K\"onig--Egerv\'ary.

Let
\[
        d_1\leq d_2\leq \cdots \leq d_n
\]
be the degree sequence of $G$.  The annihilation number of $G$, introduced by Pepper \cite{Pepper2004,Pepper2009}, is
\[
        a(G)=\max\left\{k: \sum_{i=1}^k d_i\leq m(G)\right\}.
\]
Equivalently, an annihilating set is a set $A\subseteq V(G)$ such that $\sum_{v\in A}d(v)\leq m(G)$, and $a(G)$ is the maximum size of an annihilating set.  This equivalence follows because, among all $k$-vertex subsets, the sum of the $k$ smallest degrees is the minimum possible degree sum.  Thus, an annihilating set need not literally consist of the first $k$ vertices in a degree ordering; the ordered degree sequence merely computes the largest possible cardinality.  Every independent set $I$ is annihilating, because each edge of $G$ has at most one endpoint in $I$, and hence
\[
        \sum_{v\in I}d(v)=e(I,V(G)\setminus I)\leq m(G).
\]
Consequently,
\[
        \alpha(G)\leq a(G).
\]
The present paper studies how large the difference $a(G)-\alpha(G)$ can be under natural structural hypotheses.
The work grew out of our earlier arXiv preprint \cite{KadrawiLevit2023}, which initiated the annihilation-gap program by proving matching-number bounds for trees, bipartite graphs, and K\"onig--Egerv\'ary graphs using annihilating decompositions.  Here we revisit that program, sharpen the statements to their correct integral and exact forms, add equality certificates and new extremal families, prove an exact arbitrary-graph matching-number bound, and apply the resulting theory to a TxGraffiti annihilation-residue conjecture.

Automated conjecturing supplies an additional motivation.  Fajtlowicz's Graffiti program generated graph-theoretic inequalities by filtering table-true relations among invariants \cite{Fajtlowicz1988}; see also DeLaVi\~na's historical account and the modern Dalmatian-heuristic treatment of Larson and Van Cleemput \cite{DeLaVina2005,LarsonVanCleemput2016}.  TxGraffiti continues this line with finite snapshot tables, optimization-based fitting, and Dalmatian-style redundancy filtering \cite{Davila2026}.  In this terminology, the residue inequality used below is part of the older Graffiti tradition, while the final theorem gives an annihilating-decomposition proof of a TxGraffiti-generated inequality involving $a(G)$, $\res(G)$, and $\Delta(G)$.

The annihilation number has been studied as a computable upper bound for independence and as a companion parameter for domination-type invariants; see, for example, \cite{Amjadi2015,AramKhoeilarSheikholeslamiVolkmann2018,BujtasJakovac,DehgardiNorouzianSheikholeslami2013,DehgardiSheikholeslamiKhodkar2013,DehgardiSheikholeslamiKhodkar2014,DesormeauxHaynesHenning2013,HuaXuHua2023,Jakovac2019,NingLuWang2019}.  Related degree-sequence viewpoints for independence and nullity appear in \cite{GentnerHenningRautenbach,JaumeaMolina2018}.  The equality problem $a(G)=\alpha(G)$ has also received sustained attention, including work of Larson and Pepper \cite{LarsonPepper2011}, Levit and Mandrescu \cite{LevitMandrescu2020,LevitMandrescu2022}, Hiller \cite{Hiller2023}, and Rauch and Rautenbach \cite{RauchRautenbach2023}.

Our final theorem connects the annihilation number with a second Graffiti-inspired invariant, the Havel--Hakimi residue.  Favaron, Mah\'eo and Sacl\'e \cite{FavaronMaheoSacle1991} introduced and studied the residue and proved that it is bounded above by $\alpha(G)$; Griggs and Kleitman \cite{GriggsKleitman1994} later gave a short proof.  Recently, TxGraffiti generated the conjecture
\[
        \alpha(G)\geq \frac{a(G)+\res(G)}{\Delta(G)}
\]
for connected graphs of order at least three \cite{Davila2026}.  During the preparation of this manuscript, Gupta \cite{Gupta2026} proved the same conjecture independently by a different and stronger route, using a Caro--Wei bound for the annihilation number.  We retain the result here because our proof is independent and follows from the annihilating-decomposition and matching-number methods developed in this paper.  We also use Gupta's estimate together with our exact matching-number theorem to obtain a combined computable two-sided bracket for the independence number.  A comparison appears in Section~\ref{section:txgraffiti}.

\subsection*{Contribution and novelty calibration}
The first three structural classes below were already central in the earlier Kadrawi--Levit preprint \cite{KadrawiLevit2023}.  The present version refines that starting point and adds the general matching theorem, equality-structure information, and the TxGraffiti application.  The paper is organized around five main contributions.  First, the arbitrary-graph extremal problem with fixed matching number is solved in a closed arithmetic form.  Second, K\"onig--Egerv\'ary graphs receive a sharp matching-number bound with equality for every prescribed matching number.  Third, the forest estimate is sharpened to its correct integral form and is shown to be best possible for every matching number.  Fourth, the bipartite non-tree estimate is stated in sharp integer form and is accompanied by equality examples, including a connected six-vertex extremal graph and an infinite connected equality family.  Fifth, the TxGraffiti annihilation-residue inequality is proved by an independent annihilating-decomposition argument and presented as a machine-conjecture-to-theorem case study.  In the same section, Gupta's Caro--Wei estimate is sharpened to an exact slack identity, its equality cases are identified in degree-sequence form, and the refinement is combined with our exact matching theorem to give a combined polynomial-time computable bracket for $\alpha(G)$ and a Gupta--residue upper bound on the same annihilation gap studied throughout the paper.  A separate sequel studies the fixed-matching $K_4$-free problem by finite matched cores and bounded blow-ups of type graphs.

\subsection*{Proof architecture}
The proofs are organized around one recurring device.  Given an annihilating decomposition $V(G)=A\dot\cup B$, the annihilation condition gives
\[
        \sum_{v\in A}d(v)\leq m(G),
\]
Consequently, Lemma~\ref{lem:EALEB} gives the edge-budget inequality
\[
        e(A)\leq e(B).
\]
The low-degree side $A$ controls the gap because
\[
        a(G)-\alpha(G)=|A|-\alpha(G)
        \leq |A|-\alpha(G[A])=\tau(G[A]),
\]
while the complement $B$ limits how dense $G[A]$ can be through $e(A)\leq e(B)$.  Thus, $A$ supplies the vertex-cover term in the gap, and $B$ supplies the edge budget that permits $A$ to be annihilating.  The paper uses this device in five increasingly global ways:
\begin{enumerate}
\item for forests, the acyclic structure bounds $e(A)$ and $e(B)$ sharply;
\item for bipartite graphs, the K\"onig--Egerv\'ary theorem converts $|A|-\alpha(G[A])$ into a matching quantity;
\item for K\"onig--Egerv\'ary graphs, the identity $n=\alpha+\mu$ turns bounds on $a$ into matching-number estimates;
\item for arbitrary graphs, the same inequalities reduce the problem to a one-variable optimization in $c=|B|$, giving the closed form displayed in Theorem A;
\item for the TxGraffiti inequality, the decomposition gives the auxiliary estimate
\[
        a(G)\leq (\Delta(G)-1)\alpha(G),
\]
and the residue inequality $\res(G)\leq\alpha(G)$ completes the proof.
\end{enumerate}
Thus, the TxGraffiti application is not isolated; it is the endpoint of the same annihilating-decomposition method used throughout the paper.

\subsection*{Main theorem overview}
For quick reference, we record the strongest results of the paper.

\medskip
\noindent\textbf{Theorem A: exact arbitrary-graph matching bound.}  If $G$ is any finite simple graph with $\mu(G)\geq 1$, then
\[
        a(G)-\alpha(G)\leq 2\mu(G)+1-\ceil{\sqrt{6\mu(G)}}.
\]
This bound is exact for every prescribed positive matching number.

\medskip
\noindent\textbf{Theorem B: exact K\"onig--Egerv\'ary bound.}  If $G$ is K\"onig--Egerv\'ary, then
\[
        a(G)-\alpha(G)\leq \mu(G)-\ceil{\frac{\sqrt{8\mu(G)+1}-1}{2}}.
\]
This bound is exact for every prescribed value of $\mu(G)$; for $\mu(G)\geq 2$ it is attained by connected non-bipartite K\"onig--Egerv\'ary graphs.

\medskip
\noindent\textbf{Theorem C: sharp forest bound.}  If $F$ is a forest with at least one edge, then
\[
        a(F)-\alpha(F)\leq \floor{\frac{\mu(F)-1}{2}},
\]
and equality is attained by a tree for every prescribed positive matching number.

\medskip
\noindent\textbf{Theorem D: sharp bipartite bound, sharp on non-trees.}  If $G$ is bipartite, then
\[
        a(G)-\alpha(G)\leq \floor{2+\mu(G)-2\sqrt{1+\mu(G)}}.
\]
Equivalently,
\[
        a(G)-\alpha(G)\leq 2+\mu(G)-\ceil{2\sqrt{1+\mu(G)}}.
\]
These two forms are the same because $2+\mu(G)$ is an integer.  The equality mechanism for the non-forest case is certified structurally; equality occurs already on six vertices among connected bipartite non-trees, and an infinite connected equality family is given.

\medskip
\noindent\textbf{Theorem E: TxGraffiti annihilation-residue inequality.}  If $G$ is connected and $n(G)\geq 3$, then
\[
        \alpha(G)\geq \frac{a(G)+\res(G)}{\Delta(G)}.
\]
Both hypotheses are necessary.

\subsection*{Dependency map and boundary checks}
The main proofs use the following dependency chain.  This table is included to make the logical structure easy to audit.
\begin{center}
\small
\begin{tabularx}{\textwidth}{@{}p{3.4cm}X@{}}
\toprule
Input & Used for \\
\midrule
Sorted-degree definition of $a(G)$ and annihilating sets & The basic decomposition $V(G)=A\dot\cup B$ and the edge-budget inequality $e(A)\le e(B)$. \\
Lemma~\ref{lem:EALEB} & Forest, bipartite, K\"onig--Egerv\'ary and general matching bounds. \\
K\"onig--Egerv\'ary theorem for bipartite graphs & Forest and bipartite estimates. \\
Lemmas~\ref{lem:tau_edges_matching} and~\ref{lem:e_minus_matching_complete} & Exact arbitrary-graph fixed-matching bound. \\
Brooks' theorem and $\operatorname{res}(G)\le\alpha(G)$ & TxGraffiti application. \\
\bottomrule
\end{tabularx}
\end{center}

The following small graphs are the boundary tests that the statements must pass.
\begin{center}
\small
\begin{tabularx}{\textwidth}{@{}p{2.3cm}X@{}}
\toprule
Graph & Role in the paper \\
\midrule
$K_1$ & Shows why the real-valued forest theorem must exclude edgeless forests. \\
$K_2$ & Shows the hypothesis $n(G)\ge3$ is necessary in the TxGraffiti theorem. \\
$P_3$ & Equality witness for the TxGraffiti theorem. \\
$C_3$ & Small non-bipartite equality witness in the TxGraffiti theorem and the base odd-cycle case in Lemma~\ref{lem:a_delta_alpha}. \\
$C_3\cup K_2$ & Shows connectedness is necessary in the TxGraffiti theorem. \\
$K_4$ & Small sharp witness for both the general matching bound at $\mu=2$ and the TxGraffiti theorem. \\
\bottomrule
\end{tabularx}
\end{center}

\subsection*{Sharpness guide}
The following table records the main sharpness and equality witnesses used later.  It is intended as a quick map for the reader; the constructions and proofs appear in the relevant sections.  In the Bound column, $\Gamma(\mu)$, $s(\mu)$ and $f_{\mathrm{bip}}(\mu)$ denote the expressions displayed in the theorem overview above.
\begin{table}[H]
\centering
\caption{Sharpness and equality witnesses for the main theorems.}
\label{tab:sharpness-guide}
\scriptsize
\renewcommand{\arraystretch}{1.15}
\begin{tabularx}{\linewidth}{p{0.20\linewidth}p{0.34\linewidth}X}
\toprule
Result & Bound & Sharpness or equality witnesses \\ \midrule
General graphs & $a-\alpha\leq \Gamma(\mu)$ & Exact for every $\mu$ by the clique-budget construction in Theorem~\ref{thm:general_exact}; Proposition~\ref{prop:general_location} gives necessary equality structure. \\
K\"onig--Egerv\'ary graphs & $a-\alpha\leq \mu-s(\mu)$ & Exact for every $\mu$; for $\mu\geq2$ equality occurs in connected non-bipartite K\"onig--Egerv\'ary graphs. \\
Forests & $a-\alpha\leq \lfloor(\mu-1)/2\rfloor$ & The trees $T_s$ attain equality for every positive $s=\mu$. \\
Bipartite graphs & $a-\alpha\leq f_{\mathrm{bip}}(\mu)$ & The bound holds for all bipartite graphs and is sharp on connected bipartite non-trees; equality occurs on the six-vertex graph with edges $02,03,04,12,13,35$ and on the family $H_r$. \\
TxGraffiti application & $\alpha\geq(a+\res)/\Delta$ & Equality occurs for $P_3,C_3,C_4,K_4$; $K_2$ and $C_3\cup K_2$ show that the hypotheses are necessary. \\ \bottomrule
\end{tabularx}
\end{table}

\section{Preliminaries}

For $X\subseteq V(G)$, write $G[X]$ for the subgraph induced by $X$, and write $e(X)=m(G[X])$.  If $X,Y\subseteq V(G)$ are disjoint, write $e(X,Y)$ for the number of edges with one endpoint in $X$ and one endpoint in $Y$.

\begin{definition}
An \emph{annihilating decomposition} of a graph $G$ is a partition $\langle A,B\rangle$ of $V(G)$ such that $A$ is a maximum annihilating set and $B=V(G)\setminus A$.  Thus, $|A|=a(G)$ and $|B|=n(G)-a(G)$.
\end{definition}

\begin{lemma}\label{lem:EALEB}
If $\langle A,B\rangle$ is an annihilating decomposition of $G$, then
\[
        e(A)\leq e(B).
\]
\end{lemma}

\begin{proof}
Since $A$ is annihilating, $\sum_{v\in A}d(v)\leq m(G)$.  Hence
\[
        \sum_{v\in B}d(v)=2m(G)-\sum_{v\in A}d(v)\geq m(G),
\]
so $\sum_{v\in A}d(v)\leq \sum_{v\in B}d(v)$.  But
\[
        \sum_{v\in A}d(v)=2e(A)+e(A,B),\qquad
        \sum_{v\in B}d(v)=2e(B)+e(A,B).
\]
Cancelling $e(A,B)$ gives $e(A)\leq e(B)$.
\end{proof}

\begin{lemma}\label{lem:bipartiteA}
Let $G$ be bipartite, and let $\langle A,B\rangle$ be an annihilating decomposition of $G$.  Then
\[
        a(G)-\alpha(G)\leq e(A).
\]
\end{lemma}

\begin{proof}
The graph $G[A]$ is bipartite.  Hence, by the K\"onig--Egerv\'ary theorem for bipartite graphs \cite{Egervary1931,Koenig1931,West2001},
\[
        |A|=\alpha(G[A])+\mu(G[A]).
\]
Since $\alpha(G)\geq \alpha(G[A])$ and $\mu(G[A])\leq e(A)$, we obtain
\[
        a(G)-\alpha(G)=|A|-\alpha(G)
        \leq |A|-\alpha(G[A])
        =\mu(G[A])
        \leq e(A).
\]
\end{proof}

\begin{lemma}\label{lem:KE_mu_le_alpha}
If $G$ is K\"onig--Egerv\'ary, then $\mu(G)\leq \alpha(G)$.
\end{lemma}

\begin{proof}
Since every matching has at most $n(G)/2$ edges and $\alpha(G)=n(G)-\mu(G)$, we have
\[
        \mu(G)\leq \frac{n(G)}2\leq n(G)-\mu(G)=\alpha(G).
\]
\end{proof}

\begin{lemma}\label{lem:a_nminus1}
Let $G$ be a graph with at least one edge.  Then $a(G)=n(G)-1$ if and only if all edges of $G$ are incident with one vertex.  Equivalently, $G$ is a star together with possibly some isolated vertices.
\end{lemma}

\begin{proof}
Let $d_1\leq \cdots \leq d_n$ be the degree sequence of $G$.  We have $a(G)=n(G)-1$ if and only if
\[
        \sum_{i=1}^{n-1}d_i\leq m(G).
\]
Since $\sum_{i=1}^{n}d_i=2m(G)$, this is equivalent to $d_n\geq m(G)$.  If $v$ has degree $d_n$, then $m(G)=d(v)+m(G-v)$, so $d_n\geq m(G)$ forces $m(G-v)=0$.  Thus, every edge is incident with $v$.  The converse is immediate from the degree sequence of a star plus isolated vertices.
\end{proof}

\section{Forests and trees}\label{section:trees}

The edgeless forest is an exceptional boundary case for a matching-number formulation: if $F$ is edgeless, then $a(F)=\alpha(F)=n(F)$ and $\mu(F)=0$.  In particular, the one-vertex tree would make the real-valued right-hand side equal to $-1/2$.  Therefore, the correct statement starts with forests having at least one edge.

\begin{theorem}\label{thm:tree}
If $F$ is a forest with at least one edge, then
\[
        a(F)-\alpha(F)\leq \frac{\mu(F)-1}{2}.
\]
Consequently,
\[
        a(F)-\alpha(F)\leq \floor{\frac{\mu(F)-1}{2}}.
\]
\end{theorem}

\begin{proof}
Let $\langle A,B\rangle$ be an annihilating decomposition of $F$.  Since $F$ has at least one edge, $a(F)<n(F)$, and hence, $B\neq\emptyset$.  By Lemmas \ref{lem:EALEB} and \ref{lem:bipartiteA},
\[
        a(F)-\alpha(F)\leq e(A)\leq e(B).
\]
The graph $F[B]$ is a forest on the nonempty vertex set $B$, so $e(B)\leq |B|-1$.  Therefore,
\[
        a(F)-\alpha(F)\leq |B|-1=n(F)-a(F)-1.
\]
Since every forest is bipartite, the K\"onig--Egerv\'ary theorem gives $n(F)=\alpha(F)+\mu(F)$.  Thus
\[
        a(F)-\alpha(F)
        \leq \alpha(F)+\mu(F)-a(F)-1,
\]
and hence
\[
        2(a(F)-\alpha(F))\leq \mu(F)-1.
\]
This proves the asserted inequality.  The floor version follows because the left-hand side is an integer.
\end{proof}

\begin{proposition}[Forest equality certificate]\label{prop:forest_equality_certificate}
Let $F$ be a forest with at least one edge, let $\langle A,B\rangle$ be an annihilating decomposition, and put $x=a(F)-\alpha(F)$.  If equality holds in the real-valued bound
\[
        x=\frac{\mu(F)-1}{2},
\]
then $\mu(F)$ is odd and all inequalities in the proof of Theorem \ref{thm:tree} are tight.  In particular,
\[
        x=e(A)=e(B)=|B|-1,
\]
$F[B]$ is a tree, and $F[A]$ is a matching of size $x$ together with isolated vertices.
\end{proposition}

\begin{proof}
The proof of Theorem \ref{thm:tree} gives
\[
        x\leq e(A)\leq e(B)\leq |B|-1
\]
and, since $F$ is K\"onig--Egerv\'ary,
\[
        |B|=n(F)-a(F)=\mu(F)-x.
\]
Thus, $x\leq \mu(F)-x-1$.  Equality in $x=(\mu(F)-1)/2$ forces equality throughout the displayed chain.  Hence, $e(B)=|B|-1$, so $F[B]$ is a tree.  Equality in Lemma \ref{lem:bipartiteA} gives $|A|-\alpha(F[A])=\mu(F[A])=e(A)$; hence, every edge of $F[A]$ is a component of a matching, and $F[A]$ is a matching plus isolated vertices.
\end{proof}

\begin{remark}\label{rem:forest_real_vs_integral}
Proposition~\ref{prop:forest_equality_certificate} characterizes equality in the real-valued estimate
\(a(F)-\alpha(F)\le (\mu(F)-1)/2\).  When \(\mu(F)\) is even, equality in the integral floor bound is a separate sharpness question; Proposition~\ref{prop:tree_sharp} supplies the extremal trees for every positive value of \(\mu(F)\).
\end{remark}

\begin{proposition}\label{prop:tree_sharp}
For every positive integer $s$, there is a tree $T_s$ with $\mu(T_s)=s$ such that
\[
        a(T_s)-\alpha(T_s)=\floor{\frac{s-1}{2}}.
\]
Thus, the integral forest/tree bound is sharp for every prescribed positive matching number.  The real-valued bound itself is attained exactly by this family when $s$ is odd.
\end{proposition}

\begin{proof}
For $s=1$, take $T_1=K_2$.  Now let $s\geq 2$.  Let $T_s$ be the spider with one arm of length $1$ and $s-1$ arms of length $2$: it has a root $r$, one leaf adjacent directly to $r$, and paths $r u_i w_i$ for $1\leq i\leq s-1$.

The $s-1$ vertices $w_i$ together with the direct leaf form an independent set of size $s$, and matching each $u_i$ to $w_i$ together with matching $r$ to the direct leaf gives a matching of size $s$.  Since $T_s$ is bipartite, it is K\"onig--Egerv\'ary, and hence, $\alpha(T_s)=\mu(T_s)=s$.

The degree sequence consists of $s$ vertices of degree $1$, $s-1$ vertices of degree $2$, and one root of degree $s$.  Also $m(T_s)=2s-1$.  After all $s$ leaves are selected, the remaining degree budget is $s-1$, so precisely $\floor{(s-1)/2}$ vertices of degree $2$ can be added to a maximum annihilating set.  Therefore,
\[
        a(T_s)=s+\floor{\frac{s-1}{2}},
\]
and the claimed equality follows.
\end{proof}

\begin{figure}[H]
\centering
\begin{tikzpicture}[scale=0.95, every node/.style={circle,draw,inner sep=1.3pt,minimum size=5pt}, every path/.style={line width=0.4pt}]
\node (r) at (0,0) {$r$};
\node (l) at (0,-1.25) {};
\draw (r)--(l);
\node (u1) at (-3,1.1) {};
\node (w1) at (-4.1,1.8) {};
\draw (r)--(u1)--(w1);
\node (u2) at (-1.5,1.25) {};
\node (w2) at (-2.2,2.25) {};
\draw (r)--(u2)--(w2);
\node[draw=none] (dots) at (0.3,1.55) {$\cdots$};
\node (u3) at (2.1,1.25) {};
\node (w3) at (3.0,2.15) {};
\draw (r)--(u3)--(w3);
\node[draw=none] at (0,-1.75) {one arm of length $1$};
\node[draw=none] at (0,2.75) {$s-1$ arms of length $2$};
\end{tikzpicture}
\caption{The tree $T_s$ used in Proposition~\ref{prop:tree_sharp}: one arm has length $1$, and the remaining $s-1$ arms have length $2$.}
\label{fig:Ts}
\end{figure}
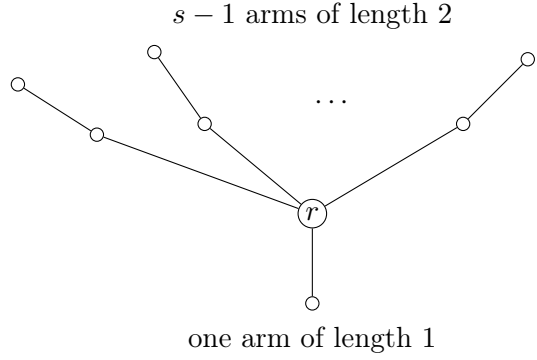

\begin{corollary}\label{cor:tree_alpha}
If $F$ is a forest with at least one edge, then
\[
        a(F)\leq \frac{3\alpha(F)-1}{2}.
\]
\end{corollary}

\begin{proof}
By Theorem \ref{thm:tree} and Lemma \ref{lem:KE_mu_le_alpha},
\[
        a(F)-\alpha(F)\leq \frac{\mu(F)-1}{2}\leq \frac{\alpha(F)-1}{2}.
\]
Rearranging gives the desired inequality.
\end{proof}

\section{Bipartite graphs}\label{section:bipartite}
The bound in this section holds for every bipartite graph.  The non-tree hypothesis enters only in the sharpness discussion, where equality is realized by connected bipartite non-trees.

\begin{theorem}\label{thm:bipartite}
If $G$ is bipartite, then
\[
        a(G)-\alpha(G)\leq 2+\mu(G)-2\sqrt{1+\mu(G)}.
\]
Equivalently, since $a(G)-\alpha(G)$ is an integer,
\[
        a(G)-\alpha(G)\leq \floor{2+\mu(G)-2\sqrt{1+\mu(G)}}.
\]
The same integer bound may also be written in the ceiling form
\[
        a(G)-\alpha(G)\leq 2+\mu(G)-\ceil{2\sqrt{1+\mu(G)}}.
\]
The equivalence of the two forms uses that $2+\mu(G)$ is an integer.
\end{theorem}

\begin{proof}
Let $\langle A,B\rangle$ be an annihilating decomposition of $G$, and put
\[
        x=a(G)-\alpha(G).
\]
By Lemmas \ref{lem:EALEB} and \ref{lem:bipartiteA},
\[
        x\leq e(A)\leq e(B).
\]
Since $G[B]$ is bipartite on $|B|$ vertices,
\[
        e(B)\leq \frac{|B|^2}{4}.
\]
Also $G$ is K\"onig--Egerv\'ary, so $n(G)=\alpha(G)+\mu(G)$, and hence
\[
        |B|=n(G)-a(G)=\mu(G)-x.
\]
Thus
\[
        x\leq \frac{(\mu(G)-x)^2}{4}.
\]
This is equivalent to
\[
        x^2-(2\mu(G)+4)x+\mu(G)^2\geq 0.
\]
The two roots of the corresponding quadratic are
\[
        2+\mu(G)\pm 2\sqrt{1+
        \mu(G)}.
\]
Since $x\leq \mu(G)$ for K\"onig--Egerv\'ary graphs, $x$ cannot lie above the larger root.  Therefore,
\[
        x\leq 2+
        \mu(G)-2\sqrt{1+
        \mu(G)},
\]
as required.
\end{proof}

\begin{proposition}[Bipartite equality certificate]\label{prop:bipartite_equality_certificate}
Let $G$ be a bipartite graph, let $\langle A,B\rangle$ be an annihilating decomposition, and put $\mu=\mu(G)$ and $x=a(G)-\alpha(G)$.  If equality holds in the real-valued inequality of Theorem \ref{thm:bipartite}, then $1+\mu$ is a square.  Writing $q=\sqrt{1+\mu}$, one has
\[
        \mu=q^2-1,
        \qquad
        x=(q-1)^2,
        \qquad
        |B|=2(q-1).
\]
Moreover,
\[
        e(A)=e(B)=x,
\]
$G[B]$ is the complete bipartite graph $K_{q-1,q-1}$, $G[A]$ is a matching of size $x$ together with isolated vertices, and
\[
        \alpha(G)=\alpha(G[A]).
\]
Conversely, if a bipartite graph has an annihilating decomposition satisfying these displayed conditions for some integer $q\geq 2$, then it attains equality in the real-valued inequality of Theorem \ref{thm:bipartite}.
\end{proposition}

\begin{proof}
The proof of Theorem \ref{thm:bipartite} gives
\[
        x\leq e(A)\leq e(B)\leq \frac{|B|^2}{4},
        \qquad
        |B|=\mu-x.
\]
Equality in the final bound is equivalent to equality in
\[
        x\leq \frac{(\mu-x)^2}{4}.
\]
Thus, $x=2+\mu-2\sqrt{1+\mu}$.  Since $x$ is an integer, $q=\sqrt{1+\mu}$ is an integer.  Hence
\[
        \mu=q^2-1,
\]
\[
        x=2+(q^2-1)-2q=(q-1)^2,
\]
and
\[
        |B|=\mu-x=(q^2-1)-(q-1)^2=2(q-1).
\]
  Equality also forces
\[
        x=e(A)=e(B)=\frac{|B|^2}{4}.
\]
The last equality is the extremal equality case for bipartite graphs, so $G[B]$ is $K_{q-1,q-1}$.  Equality in Lemma \ref{lem:bipartiteA} forces $\alpha(G)=\alpha(G[A])$ and $\mu(G[A])=e(A)$; hence, the edges of $G[A]$ are pairwise disjoint.  The converse follows by reversing the same chain of equalities.
\end{proof}

\begin{theorem}[Connected bipartite real-equality template]\label{thm:bipartite_connected_classification}
Let $G$ be a connected bipartite graph.  Then $G$ attains equality in the real-valued inequality of Theorem~\ref{thm:bipartite} if and only if there is an integer $q\geq2$ and a partition $V(G)=A\dot\cup B$ with the following properties.  First, $B=B_1\dot\cup B_2$, where $|B_1|=|B_2|=q-1$, and
\[
        G[B]\cong K_{q-1,q-1}
\]
with partite classes $B_1$ and $B_2$.  Second, $G[A]$ is a matching of $(q-1)^2$ edges together with isolated vertices.  Third, $A$ is a maximum annihilating set of $G$.  Fourth,
\[
        \alpha(G)=\alpha(G[A]).
\]
Finally, the resulting graph is connected.

Equivalently, the fourth condition may be replaced by the following explicit cross-edge condition: for every set $S\subseteq B_i$, $i\in\{1,2\}$,
\[
        |S|+\alpha\bigl(G[A\setminus N_G(S)]\bigr)\leq \alpha(G[A]).
\]
For every graph satisfying these conditions one has
\[
        \mu(G)=q^2-1,
        \qquad
        a(G)-\alpha(G)=(q-1)^2.
\]
\end{theorem}

\begin{proof}
Suppose first that $G$ is connected and attains equality.  Proposition~\ref{prop:bipartite_equality_certificate} applied to an annihilating decomposition $\langle A,B\rangle$ gives an integer $q\geq2$ such that $G[B]\cong K_{q-1,q-1}$, $G[A]$ is a matching of $(q-1)^2$ edges together with isolated vertices, $A$ is maximum annihilating, and $\alpha(G)=\alpha(G[A])$.  The graph is connected by hypothesis.

It remains only to justify the stated cross-edge reformulation.  Since $G[B]$ is complete bipartite, every independent set that meets $B$ meets at most one of $B_1$ and $B_2$.  If its intersection with $B_i$ is $S$, then its remaining vertices lie in $A\setminus N_G(S)$, and hence its size is at most
\[
        |S|+\alpha\bigl(G[A\setminus N_G(S)]\bigr).
\]
Thus no independent set using vertices of $B$ is larger than $\alpha(G[A])$ precisely when the displayed inequality holds for all such $S$.  Independent sets contained in $A$ are already bounded by $\alpha(G[A])$.

Conversely, suppose that a connected bipartite graph has such a partition.  Since $A$ is a maximum annihilating set, $a(G)=|A|$.  Since $G[A]$ is a matching of $(q-1)^2$ edges together with isolated vertices,
\[
        |A|-\alpha(G[A])=(q-1)^2.
\]
The cross-edge condition gives $\alpha(G)=\alpha(G[A])$.  Therefore
\[
        a(G)-\alpha(G)=|A|-\alpha(G[A])=(q-1)^2.
\]
Because $G$ is bipartite, it is K\"onig--Egerv\'ary, and so
\[
        \mu(G)=n(G)-\alpha(G)=|B|+|A|-\alpha(G[A])=2(q-1)+(q-1)^2=q^2-1.
\]
Hence
\[
        2+\mu(G)-2q=2+(q^2-1)-2q=(q-1)^2,
\]
and equality in Theorem~\ref{thm:bipartite} follows.
\end{proof}

\begin{remark}
Theorem~\ref{thm:bipartite_connected_classification} is a structural template for the connected bipartite real-equality cases.  The remaining freedom is exactly the freedom to add bipartition-preserving cross-edges between the matching-plus-isolates side $A$ and the balanced complete bipartite side $B$, subject to the maximum-annihilating condition, the independent-set condition displayed in the theorem, and connectedness.  Thus the theorem is not merely a list of examples, but neither does it enumerate every cross-edge pattern separately.
\end{remark}

\begin{proposition}\label{prop:six_vertex}
There is a connected bipartite non-tree on six vertices attaining equality in Theorem \ref{thm:bipartite}.  Moreover, six is the smallest possible order of a connected bipartite non-tree attaining equality.
\end{proposition}

\begin{figure}[H]
\centering
\begin{tikzpicture}[scale=0.95, every node/.style={circle,draw,inner sep=1.2pt,minimum size=16pt}]
\node (v0) at (0,1.4) {$0$};
\node (v1) at (1.3,1.4) {$1$};
\node (v5) at (2.6,1.4) {$5$};
\node (v2) at (0,0) {$2$};
\node (v3) at (1.3,0) {$3$};
\node (v4) at (2.6,0) {$4$};
\draw (v0)--(v2); \draw (v0)--(v3); \draw (v0)--(v4);
\draw (v1)--(v2); \draw (v1)--(v3);
\draw (v5)--(v3);
\end{tikzpicture}
\caption{The six-vertex connected bipartite equality graph, with bipartition $\{0,1,5\}\cup\{2,3,4\}$.}
\label{fig:six-vertex-bipartite}
\end{figure}
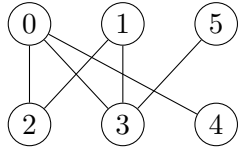

\begin{proof}
Let $G$ have vertex set $\{0,1,2,3,4,5\}$ and edge set
\[
        \{02,03,04,12,13,35\}.
\]
This graph is connected and bipartite, with bipartition
\[
        \{0,1,5\}\cup \{2,3,4\}.
\]
It is not a tree, since $0,2,1,3,0$ is a cycle.  Its degree sequence is
\[
        1,1,2,2,3,3,
\]
and $m(G)=6$.  Hence, $a(G)=4$.  The edges $04,12,35$ form a perfect matching, so $\mu(G)=3$, and since $G$ is bipartite, $\alpha(G)=n(G)-\mu(G)=3$.  Therefore,
\[
        a(G)-\alpha(G)=1.
\]
On the other hand,
\[
        2+
        \mu(G)-2\sqrt{1+
        \mu(G)}=2+3-2\sqrt4=1.
\]
Thus, equality holds.

For minimality, let $H$ be a connected bipartite non-tree attaining equality.  Then $H$ contains an even cycle, so $\mu(H)\geq 2$.  Equality forces $2+
\mu(H)-2\sqrt{1+
\mu(H)}$ to be an integer, and hence, $1+
\mu(H)$ must be a square.  If $n(H)\leq 5$, then $\mu(H)\leq 2$, and therefore, $\mu(H)=2$; but then $1+\mu(H)=3$ is not a square.  Hence, no such graph has fewer than six vertices.
\end{proof}

\begin{proposition}\label{prop:connected_bipartite_family}
There is an infinite family of connected bipartite non-trees attaining equality in Theorem \ref{thm:bipartite}.
\end{proposition}

\begin{proof}
Fix an integer $r\geq 2$.  Start with a complete bipartite graph $K_{r,r}$ with core bipartition $X\cup Y$.  Attach one pendant leaf to every vertex of $X\cup Y$.  Next add vertices $b_1,\ldots,b_{r^2}$; join each $b_i$ to a fixed vertex of $X$, and attach one pendant leaf $b_i'$ to $b_i$.  This gives a connected bipartite graph $H_r$.  This family is not obtained by disjoint unions of the six-vertex graph; the complete bipartite core keeps the graph connected.  Since $r\geq 2$, the core $K_{r,r}$ contains a cycle, so $H_r$ is not a tree.

\begin{figure}[H]
\centering
\begin{tikzpicture}[
        scale=1.0,
        core/.style={circle,draw,fill=white,inner sep=1.1pt,minimum size=7pt,font=\scriptsize},
        leaf/.style={circle,draw,fill=gray!12,inner sep=0.8pt,minimum size=6pt},
        addon/.style={circle,draw,fill=gray!25,inner sep=1.0pt,minimum size=7pt,font=\scriptsize},
        faint/.style={line width=0.35pt,gray!55},
        edge/.style={line width=0.45pt}
]
\node[core] (x1) at (0,1.25) {$x_1$};
\node[core] (x2) at (0,0.45) {$x_2$};
\node[draw=none,font=\small] (xd) at (0,-0.20) {$\vdots$};
\node[core] (xr) at (0,-0.95) {$x_r$};
\node[core] (y1) at (3.60,1.25) {$y_1$};
\node[core] (y2) at (3.60,0.45) {$y_2$};
\node[draw=none,font=\small] (yd) at (3.60,-0.20) {$\vdots$};
\node[core] (yr) at (3.60,-0.95) {$y_r$};
\foreach \x in {x1,x2,xr}{\foreach \y in {y1,y2,yr}{\draw[faint] (\x)--(\y);}}
\node[draw=none,font=\small] at (1.80,1.70) {core $K_{r,r}$};
\node[draw=none,font=\small] at (-0.35,-1.35) {$X$};
\node[draw=none,font=\small] at (3.95,-1.35) {$Y$};

\foreach \x/\lx/\y in {x1/lx1/1.25,x2/lx2/0.45,xr/lxr/-0.95}{
    \node[leaf] (\lx) at (-0.85,\y) {}; \draw[edge] (\lx)--(\x);
}
\foreach \y/\ly/\yy in {y1/ly1/1.25,y2/ly2/0.45,yr/lyr/-0.95}{
    \node[leaf] (\ly) at (4.45,\yy) {}; \draw[edge] (\y)--(\ly);
}

\node[addon] (b1) at (0.70,-2.25) {$b_1$};
\node[addon] (b2) at (1.45,-2.25) {$b_2$};
\node[draw=none,font=\small] at (2.20,-2.25) {$\cdots$};
\node[addon] (br) at (2.95,-2.25) {$b_{r^2}$};
\foreach \b in {b1,b2,br}{\draw[faint] (x1)--(\b);}
\node[leaf] (bp1) at (0.70,-3.05) {}; \draw[edge] (b1)--(bp1);
\node[leaf] (bp2) at (1.45,-3.05) {}; \draw[edge] (b2)--(bp2);
\node[leaf] (bpr) at (2.95,-3.05) {}; \draw[edge] (br)--(bpr);
\node[draw=none,font=\scriptsize] at (1.80,-3.55) {$r^2$ branches $x_1 b_i b_i'$};
\end{tikzpicture}
\caption{The connected bipartite equality family $H_r$.  The core is $K_{r,r}$; every core vertex has one pendant leaf, and there are $r^2$ additional branches $x_1b_i b_i'$.}
\label{fig:Hr}
\end{figure}
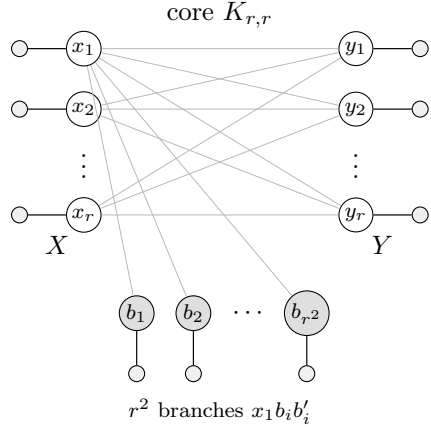

The graph has
\[
        n(H_r)=2r^2+4r
\]
vertices.  There is a perfect matching: match every core vertex to its pendant leaf, and match every $b_i$ to $b_i'$.  Thus
\[
        \mu(H_r)=\frac{n(H_r)}{2}=r^2+2r.
\]
Because $H_r$ is bipartite, it is K\"onig--Egerv\'ary, and so
\[
        \alpha(H_r)=n(H_r)-\mu(H_r)=r^2+2r.
\]

The number of edges is
\[
        m(H_r)=r^2+2r+r^2+r^2=3r^2+2r,
\]
where the four terms count respectively the edges of $K_{r,r}$, the pendant edges at the core vertices, the edges joining the $b_i$ to the core, and the pendant edges $b_ib_i'$.  There are $r^2+2r$ leaves of degree one and $r^2$ vertices $b_i$ of degree two.  The sum of their degrees is
\[
        (r^2+2r)+2r^2=3r^2+2r=m(H_r).
\]
All remaining vertices have degree at least $r+1\geq 3$, so
\[
        a(H_r)=2r^2+2r.
\]
Consequently
\[
        a(H_r)-\alpha(H_r)=r^2.
\]
Since
\[
        \mu(H_r)=r^2+2r=(r+1)^2-1,
\]
we also have
\[
        2+
        \mu(H_r)-2\sqrt{1+
        \mu(H_r)}
        =2+(r+1)^2-1-2(r+1)=r^2.
\]
Thus, equality holds for every $r\geq 2$.
\end{proof}

\begin{corollary}\label{cor:bip_alpha}
If $G$ is bipartite, then
\[
        a(G)\leq 2+2\alpha(G)-2\sqrt{1+
        \alpha(G)}.
\]
\end{corollary}

\begin{proof}
The function $f(t)=2+t-2\sqrt{1+t}$ is increasing for $t\geq 0$.  By Lemma \ref{lem:KE_mu_le_alpha}, $\mu(G)\leq \alpha(G)$.  Therefore, Theorem~\ref{thm:bipartite} gives
\[
        a(G)-\alpha(G)\leq f(\mu(G))\leq f(\alpha(G))
        =2+\alpha(G)-2\sqrt{1+\alpha(G)}.
\]
Adding $\alpha(G)$ to both sides gives the claimed inequality.
\end{proof}

\section{K\"onig--Egerv\'ary graphs}\label{section:KE}

K\"onig--Egerv\'ary graphs and their equivalent matching-cover formulations have a substantial structural and algorithmic literature.  The foundational characterization and recognition sources include Deming, Gavril, and Sterboul \cite{Deming1979,Gavril1977,Sterboul1979}.  Subsequent work has developed the structure of maximum stable sets, cores and coronas, set-and-collection lemmas for families of maximum stable sets, stability under edge operations, and critical edges; see, for example, \cite{BorosGolumbicLevit2002,LevitMandrescu2002,LevitMandrescu2003,LevitMandrescu2006,LevitMandrescu2014SetCollection}.  Further work treats common maximum-matching characterizations, critical independent sets, maximum matchings, additional characterizations, deletion-preserving versions of the K\"onig--Egerv\'ary property, and $1$-K\"onig--Egerv\'ary graphs \cite{LevitMandrescu2011CommonMatchings,LevitMandrescu2012KE,LevitMandrescu2013MaxMatchings,JardenLevitMandrescu2017,LevitMandrescu2025Deletion,LevitMandrescu2026OneKE}.  The following theorem gives a stronger K\"onig--Egerv\'ary estimate than the elementary non-bipartite bound $a(G)-\alpha(G)\leq \mu(G)-2$.

\begin{theorem}\label{thm:KE_strong}
Let $G$ be a K\"onig--Egerv\'ary graph, and define
\[
        s(\mu)=\ceil{\frac{\sqrt{8\mu+1}-1}{2}}.
\]
Then
\[
        a(G)-\alpha(G)\leq \mu(G)-s(\mu(G)).
\]
\end{theorem}

\begin{proof}
Let $\langle A,B\rangle$ be an annihilating decomposition of $G$, and put
\[
        r=|B|=n(G)-a(G).
\]
By Lemma \ref{lem:EALEB}, $e(A)\leq e(B)$.  Since $|B|=r$,
\[
        e(B)\leq \binom{r}{2}.
\]
Let $M$ be a maximum matching.  At most $r$ edges of $M$ are incident with vertices of $B$.  Every remaining edge of $M$ lies in $G[A]$, so the number of remaining edges is at most $e(A)$.  Therefore,
\[
        \mu(G)=|M|\leq r+e(A)\leq r+e(B)\leq r+\binom{r}{2}=\frac{r(r+1)}2.
\]
Hence
\[
        r\geq \ceil{\frac{\sqrt{8\mu(G)+1}-1}{2}}=s(\mu(G)).
\]
Since $G$ is K\"onig--Egerv\'ary,
\[
        a(G)-\alpha(G)=n(G)-r-\alpha(G)=\mu(G)-r.
\]
Combining the last two displays gives the result.
\end{proof}

\begin{theorem}[Sharpness of the K\"onig--Egerv\'ary bound]\label{thm:KE_strong_sharp}
For $\mu=1$ there is a connected K\"onig--Egerv\'ary graph attaining equality in Theorem \ref{thm:KE_strong}.  For every $\mu\geq 2$, there is a connected non-bipartite K\"onig--Egerv\'ary graph $G$ with $\mu(G)=\mu$ and
\[
        a(G)-\alpha(G)
        =
        \mu-
        \ceil{\frac{\sqrt{8\mu+1}-1}{2}}.
\]
Thus, Theorem \ref{thm:KE_strong} is best possible for every prescribed matching number.
\end{theorem}

\begin{proof}
For $\mu=1$, take $K_2$.  For $\mu=2$, take a triangle with one pendant leaf attached to one triangle vertex.  This graph has $\alpha=\mu=2$ and $a=2$, so equality holds.  For $\mu=3$, take a triangle and attach one pendant leaf to each triangle vertex.  Then $\alpha=\mu=3$, the degree sequence is $1,1,1,3,3,3$, and $m=6$, so $a=4$ and $a-\alpha=1$, which is the asserted value.

Now let $\mu\geq 4$, and put
\[
        r=\ceil{\frac{\sqrt{8\mu+1}-1}{2}},
        \qquad
        t=\mu-r.
\]
Then $r\geq 3$.  By definition, $r$ is the least integer satisfying
\[
        \mu\leq \binom{r+1}{2}.
\]
Thus
\[
        \binom r2<\mu\leq \binom{r+1}{2}.
\]
Equivalently,
\[
        t\leq \binom r2,
        \qquad
        r+t>\binom r2.
\]
Construct $G$ as follows.  Start with a clique $B=K_r$ on vertices $b_1,\ldots,b_r$.  Add pendant vertices $p_1,\ldots,p_r$, where $p_i$ is adjacent only to $b_i$.  Finally, add $t$ disjoint edges $u_jv_j$, $1\leq j\leq t$, and join every $u_j$ to $b_1$.

The graph is connected and non-bipartite.  It has $2r+2t=2\mu$ vertices and has a perfect matching consisting of the edges $b_ip_i$ and $u_jv_j$, so $\mu(G)=\mu$.  Since $\alpha(G)+\mu(G)\leq n(G)=2\mu$ and the set
\[
        \{p_1,\ldots,p_r\}\cup \{v_1,\ldots,v_t\}
\]
is independent of size $r+t=\mu$, we have $\alpha(G)=\mu$; hence, $G$ is K\"onig--Egerv\'ary.

The number of edges is
\[
        m(G)=\binom r2+r+2t.
\]
There are $r+t$ vertices of degree $1$, namely the $p_i$ and the $v_j$, and there are $t$ vertices $u_j$ of degree $2$.  The sum of the degrees of these $r+2t$ vertices is
\[
        r+t+2t=r+3t.
\]
Since $t\leq \binom r2$, this sum is at most $m(G)$.  Adding any clique vertex increases the sum by at least $r$, and this would exceed $m(G)$ because $r+t>\binom r2$.  Hence
\[
        a(G)=r+2t.
\]
Therefore,
\[
        a(G)-\alpha(G)=(r+2t)-(r+t)=t=\mu-r,
\]
which is the claimed equality.
\end{proof}

\medskip
The equality examples in Theorem~\ref{thm:KE_strong_sharp} attain the new exact bound of Theorem~\ref{thm:KE_strong}.  The older estimate $a(G)-\alpha(G)\leq \mu(G)-2$ for non-bipartite K\"onig--Egerv\'ary graphs is weaker once $\mu(G)\geq4$, because then $s(\mu(G))\geq3$.

\begin{corollary}\label{cor:KE_nonbip}
If $G$ is a non-bipartite K\"onig--Egerv\'ary graph, then
\[
        a(G)-\alpha(G)\leq \mu(G)-2.
\]
Moreover, if $\mu(G)\geq 4$, then
\[
        a(G)-\alpha(G)\leq \mu(G)-3.
\]
\end{corollary}

\begin{proof}
Since $G$ is non-bipartite, it cannot have an independent set of size $n(G)-1$; otherwise all edges would be incident with the one vertex outside such a set, making $G$ bipartite.  Thus, $\alpha(G)\leq n(G)-2$.  Because $G$ is K\"onig--Egerv\'ary, $\mu(G)=n(G)-\alpha(G)\geq 2$, and hence, $s(\mu(G))\geq 2$.  This gives the first inequality from Theorem \ref{thm:KE_strong}.  If $\mu(G)\geq 4$, then $s(\mu(G))\geq 3$, giving the second inequality.
\end{proof}

\begin{remark}
Lemma \ref{lem:a_nminus1} gives another quick proof of the first inequality in Corollary \ref{cor:KE_nonbip}.  A non-bipartite graph cannot be a star plus isolated vertices, so $a(G)\leq n(G)-2$.  Since $G$ is K\"onig--Egerv\'ary, $n(G)=\alpha(G)+\mu(G)$, and hence, $a(G)-\alpha(G)\leq \mu(G)-2$.
\end{remark}

\begin{proposition}\label{prop:KE_sharp}
The bound $a(G)-\alpha(G)\leq \mu(G)-2$ for non-bipartite K\"onig--Egerv\'ary graphs is attained for $\mu=2$ and for $\mu=3$.  For $\mu\geq 4$, equality in this bound is impossible.
\end{proposition}

\begin{proof}
For $\mu=2$, take a triangle with one pendant leaf attached to one triangle vertex.  Its degree sequence is $1,2,2,3$, and $m=4$, so $a=2$.  Also $\alpha=2$ and $\mu=2$, whence $a-\alpha=0=\mu-2$.

For $\mu=3$, fix $k\geq 1$ and take a triangle with vertices $x,y,z$.  Attach $k$ pendant leaves to $x$, attach $k$ pendant leaves to $y$, and attach one pendant leaf to $z$.  The graph has $2k+4$ vertices and is K\"onig--Egerv\'ary with $\alpha=2k+1$ and $\mu=3$.  Its degree sequence is
\[
        \underbrace{1,\ldots,1}_{2k+1},\;3,\;k+2,\;k+2,
\]
and $m=2k+4$.  Hence, $a=2k+2$, and so
\[
        a-\alpha=1=\mu-2.
\]

Finally, if $\mu\geq 4$, Corollary \ref{cor:KE_nonbip} gives $a-\alpha\leq \mu-3$, so equality in $a-\alpha\leq \mu-2$ is impossible.
\end{proof}

\begin{corollary}\label{cor:KE_alpha}
If $G$ is a non-bipartite K\"onig--Egerv\'ary graph, then
\[
        a(G)\leq 2\alpha(G)-2.
\]
If, in addition, $\mu(G)\geq 4$, then
\[
        a(G)\leq 2\alpha(G)-3.
\]
\end{corollary}

\begin{proof}
By Corollary~\ref{cor:KE_nonbip}, $a(G)-\alpha(G)\leq \mu(G)-2$.  Lemma~\ref{lem:KE_mu_le_alpha} gives $\mu(G)\leq \alpha(G)$, and hence, $a(G)\leq 2\alpha(G)-2$.  If $\mu(G)\geq4$, then Corollary~\ref{cor:KE_nonbip} gives $a(G)-\alpha(G)\leq \mu(G)-3$, so the same argument yields $a(G)\leq2\alpha(G)-3$.
\end{proof}

\section{General graphs with prescribed matching number}\label{section:general}

The preceding sections treat special graph classes.  For arbitrary graphs the exact extremal function in terms of the matching number alone is larger, but it has a clean closed form.  We first keep the proof in its natural optimization form and then evaluate that optimization explicitly.

For an integer $\mu\geq 0$, define
\[
        \Phi(\mu)=
        \max_{0\leq c\leq 2\mu}
        \min\left\{
        2\mu-c,
        \floor{\frac{\binom c2+\mu-\floor{c/2}}{2}}
        \right\}.
\]
Also define
\[
        \Gamma(0)=0,
        \qquad
        \Gamma(\mu)=2\mu+1-\ceil{\sqrt{6\mu}}\quad(\mu\geq 1).
\]

\begin{lemma}\label{lem:Phi_closed}
For every integer $\mu\geq 0$,
\[
        \Phi(\mu)=\Gamma(\mu).
\]
\end{lemma}

\begin{proof}
The case $\mu=0$ is immediate, so assume $\mu\geq 1$.  Put
\[
        A(c)=2\mu-c,
        \qquad
        B(c)=\floor{\frac{\binom c2+\mu-\floor{c/2}}{2}}.
\]
The function $A(c)$ is strictly decreasing, while $B(c)$ is nondecreasing.  Since $A(c)$ is integral,
\[
        B(c)\geq A(c)
\]
if and only if
\[
        \binom c2+\mu-\floor{c/2}\geq 2(2\mu-c),
\]
or equivalently
\[
        \binom c2+2c-\floor{c/2}\geq 3\mu.
\]
For every integer $c\geq 0$,
\[
        \binom c2+2c-\floor{c/2}
        =\floor{\frac{(c+1)^2}{2}}.
\]
Thus, the least $c$ for which $B(c)\geq A(c)$ is
\[
        c_\mu=\ceil{\sqrt{6\mu}}-1.
\]
For $c<c_\mu$, the minimum is $B(c)$, and by the minimality of $c_\mu$ we have
\[
        B(c)\leq B(c_\mu-1)<A(c_\mu-1)=A(c_\mu)+1.
\]
Since $B(c)$ is integral, $B(c)\leq A(c_\mu)$.  For $c\geq c_\mu$, the minimum is at most $A(c)\leq A(c_\mu)$.  Hence
\[
        \Phi(\mu)=A(c_\mu)=2\mu+1-\ceil{\sqrt{6\mu}}=\Gamma(\mu).
\]
\end{proof}

\begin{lemma}\label{lem:tau_edges_matching}
For every graph $H$,
\[
        2\tau(H)\leq m(H)+\mu(H).
\]
Equivalently,
\[
        m(H)\geq 2\tau(H)-\mu(H).
\]
\end{lemma}

\begin{proof}
Let $M$ be a maximum matching of $H$, with $|M|=\mu(H)$.  Since $M$ is maximal, every edge of $H$ has at least one endpoint incident with an edge of $M$.  Choose independently and uniformly one endpoint from each edge of $M$, and let $S$ be the set of chosen vertices.  The set $S$ covers all edges of $M$.  If an edge outside $M$ has one endpoint not incident with $M$, then it is uncovered by $S$ with probability $1/2$; if its endpoints lie on two distinct edges of $M$, then it is uncovered with probability $1/4$.  In all cases, the probability that a non-matching edge is uncovered is at most $1/2$.

Now add one endpoint of every uncovered non-matching edge to $S$.  The resulting set is a vertex cover.  Its expected size is at most
\[
        \mu(H)+\frac{m(H)-\mu(H)}2
        =\frac{m(H)+\mu(H)}2.
\]
Therefore, $H$ has a vertex cover of size at most $(m(H)+\mu(H))/2$, which is the desired inequality.
\end{proof}

\begin{lemma}\label{lem:e_minus_matching_complete}
If $H$ has $c$ vertices, then
\[
        m(H)-\mu(H)\leq \binom c2-\floor{c/2}.
\]
\end{lemma}

\begin{proof}
Add the missing edges of $H$ one at a time until the complete graph $K_c$ is obtained.  Adding one edge increases the number of edges by $1$ and increases the matching number by at most $1$.  Thus, the quantity $m(H)-\mu(H)$ never decreases during this process.  Hence
\[
        m(H)-\mu(H)
        \leq m(K_c)-\mu(K_c)
        =\binom c2-\floor{c/2}.
\]
\end{proof}

\begin{theorem}[Exact general matching bound]\label{thm:general_exact}
For every finite simple graph $G$,
\[
        a(G)-\alpha(G)\leq \Gamma(\mu(G)).
\]
Equivalently, if $\mu(G)\geq 1$, then
\[
        a(G)-\alpha(G)
        \leq
        2\mu(G)+1-\ceil{\sqrt{6\mu(G)}}.
\]
Moreover, for every integer $\mu\geq 0$, there exists a graph $G$ with $\mu(G)=\mu$ and
\[
        a(G)-\alpha(G)=\Gamma(\mu).
\]
\end{theorem}

\begin{proof}
Let $\mu=\mu(G)$, let $\langle A,B\rangle$ be an annihilating decomposition, and put
\[
        x=a(G)-\alpha(G),
        \qquad
        c=|B|=n(G)-a(G).
\]
Since the endpoints of a maximum matching form a vertex cover, $\tau(G)=n(G)-\alpha(G)\leq 2\mu$.  Hence
\[
        x+c=n(G)-\alpha(G)\leq 2\mu,
\]
and so
\begin{equation}\label{eq:general_first_bound_9}
        x\leq 2\mu-c.
\end{equation}

Let $H=G[A]$ and $J=G[B]$.  Since $\alpha(G)\geq \alpha(H)$,
\[
        x=|A|-\alpha(G)\leq |A|-\alpha(H)=\tau(H).
\]
By Lemma \ref{lem:tau_edges_matching},
\[
        e(A)=m(H)\geq 2\tau(H)-\mu(H)\geq 2x-\mu(H).
\]
The matchings of $H$ and $J$ are vertex-disjoint, and hence
\[
        \mu(H)+\mu(J)\leq \mu.
\]
Using Lemma \ref{lem:EALEB} and Lemma \ref{lem:e_minus_matching_complete}, we get
\[
\begin{aligned}
        2x
        &\leq e(A)+\mu(H)  \\
        &\leq e(B)+\mu-\mu(J) \\
        &=\mu+\bigl(e(B)-\mu(J)\bigr) \\
        &\leq \mu+\binom c2-\floor{c/2}.
\end{aligned}
\]
Thus
\begin{equation}\label{eq:general_second_bound_9}
        x\leq
        \floor{\frac{\binom c2+\mu-\floor{c/2}}{2}}.
\end{equation}
Combining \eqref{eq:general_first_bound_9} and \eqref{eq:general_second_bound_9} gives $x\leq\Phi(\mu)$, and Lemma \ref{lem:Phi_closed} gives $x\leq\Gamma(\mu)$.

It remains to prove sharpness.  The cases $\mu=0$ and $\mu=1$ are realized by an edgeless graph and by $K_3$, respectively.  Assume $\mu\geq 2$, and set
\[
        c=\ceil{\sqrt{6\mu}}-1,
        \qquad
        X=\Gamma(\mu)=2\mu-c,
        \qquad
        t=\floor{c/2},
        \qquad
        s=\mu-t.
\]
The proof of Lemma \ref{lem:Phi_closed} gives
\[
        X\leq \floor{\frac{\binom c2+s}{2}},
\]
and hence $2X-s\leq \binom c2$.  Also $s\leq X\leq 2s$.  Put
\[
        q=X-s,
        \qquad
        p=2s-X.
\]
Then $p,q\geq 0$.  Let $H$ be the disjoint union of $q$ triangles and $p$ copies of $K_2$.  Then
\[
        \tau(H)=X,
        \qquad
        \mu(H)=s,
        \qquad
        m(H)=2X-s.
\]
Let $B$ induce a clique $K_c$.  Since $m(H)\leq \binom c2$, there are enough edges in $B$ to make the vertex set of $H$ annihilating after the joining operation below.

If $p>0$, choose one of the $K_2$ components of $H$ and join every vertex of $B$ to both of its vertices.  If $p=0$, then $q>0$; choose one triangle component of $H$ and join every vertex of $B$ to all three of its vertices.  Let $G$ be the resulting graph.  The chosen component intersects every maximum independent set of $H$, so adding vertices from $B$ cannot increase the independence number.  Hence
\[
        \alpha(G)=\alpha(H)=|V(H)|-X.
\]
The matching number is
\[
        \mu(G)=s+\floor{c/2}=\mu.
\]
Indeed, if $p>0$, the joined block is obtained from $K_c\cup K_2$ by adding all edges between $K_c$ and the chosen $K_2$; it is the clique $K_{c+2}$ and has matching number
\[
        \floor{(c+2)/2}=\floor{c/2}+1,
\]
the same as $K_c\cup K_2$.  If $p=0$, then $X=2s$.  Since $X=2\mu-c$ and $s=\mu-\floor{c/2}$, we get $c=2\floor{c/2}$, so $c$ is even.  The joined block is then $K_{c+3}$, whose matching number is
\[
        \floor{(c+3)/2}=c/2+1=\floor{c/2}+1,
\]
the same as $K_c$ together with the chosen triangle.  All other components of $H$ are disjoint from the joined block, and hence, the total matching number is $s+\floor{c/2}=\mu$.

Finally, if $R$ is the number of cross-edges between $H$ and $B$, then the degree sum over $V(H)$ is $2m(H)+R$, while
\[
        m(G)=m(H)+\binom c2+R.
\]
Since $m(H)\leq\binom c2$, the set $V(H)$ is annihilating in $G$.  Therefore,
\[
        a(G)-\alpha(G)
        \geq |V(H)|-(|V(H)|-X)=X=\Gamma(\mu).
\]
The upper bound already proved forces equality.
\end{proof}

\begin{proposition}[Location of extremal decompositions]\label{prop:general_location}
Let $G$ satisfy equality in Theorem~\ref{thm:general_exact}, and put $\mu=\mu(G)\geq3$.  Let $\langle A,B\rangle$ be an annihilating decomposition, set $c=|B|$, and put
\[
        r=\ceil{\sqrt{6\mu}}.
\]
Then
\[
        c\in\{r-2,r-1\}.
\]
Consequently
\[
        n(G)-\alpha(G)=a(G)-\alpha(G)+|B|\in\{2\mu-1,2\mu\}.
\]
Equivalently, every extremal graph for the general matching bound has matching-cover defect
\[
        2\mu(G)-\tau(G)\in\{0,1\}.
\]
If $c=r-1$, then $\tau(G)=2\mu(G)$, so the endpoints of every maximum matching form a minimum vertex cover.  If $c=r-2$, then $\tau(G)=2\mu(G)-1$ and the second inequality in the proof of Theorem~\ref{thm:general_exact} is tight for this decomposition.
\end{proposition}

\begin{proof}
Let $x=a(G)-\alpha(G)$.  Since equality holds,
\[
        x=\Gamma(\mu)=2\mu+1-r.
\]
The first bound in the proof of Theorem~\ref{thm:general_exact} gives
\[
        x\leq 2\mu-c,
\]
so $c\leq r-1$.

We claim that $c\leq r-3$ is impossible.  Since $\mu\geq3$, we have $r\geq5$ and $(r-1)^2<6\mu$.  Let
\[
        B_\mu(c)=\floor{\frac{\binom c2+\mu-\floor{c/2}}{2}}.
\]
The function $B_\mu(c)$ is nondecreasing in $c$.  Hence, if $c\leq r-3$, then
\[
        B_\mu(c)\leq B_\mu(r-3).
\]
It is enough to show $B_\mu(r-3)<\Gamma(\mu)$.  Put
\[
        N=\binom{r-3}{2}+\mu-\floor{\frac{r-3}{2}}.
\]
Since
\[
        \binom{r-3}{2}+2r-2-\floor{\frac{r-3}{2}}
        \leq \frac{(r-1)^2}{2}
        <3\mu,
\]
we get
\[
        N<4\mu+2-2r=2\Gamma(\mu).
\]
Thus, $B_\mu(r-3)=\lfloor N/2\rfloor<\Gamma(\mu)$, and hence, $B_\mu(c)<\Gamma(\mu)$ for all $c\leq r-3$.  This contradicts the second bound in the proof of Theorem~\ref{thm:general_exact}, which is necessary for $x=\Gamma(\mu)$.  Therefore, $c\geq r-2$.

Combining $c\leq r-1$ and $c\geq r-2$ gives $c\in\{r-2,r-1\}$.  Since
\[
        \tau(G)=n(G)-\alpha(G)=a(G)-\alpha(G)+|B|=x+c,
\]
we obtain $\tau(G)=2\mu$ when $c=r-1$ and $\tau(G)=2\mu-1$ when $c=r-2$.  In the latter case the first bound gives only $x\leq\Gamma(\mu)+1$; hence, equality in the overall theorem forces the second bound to be tight.
\end{proof}

\begin{proposition}[Equality certificates for the general bound]\label{prop:general_certificate}
Let $G$ satisfy equality in Theorem \ref{thm:general_exact}, let $\mu=\mu(G)$, and let $\langle A,B\rangle$ be an annihilating decomposition.  Put $c=|B|$ and $x=a(G)-\alpha(G)$.  Then
\[
        x=\Gamma(\mu)
\]
and
\[
        \min\left\{2\mu-c,
        \floor{\frac{\binom c2+\mu-\floor{c/2}}{2}}\right\}=\Gamma(\mu).
\]
If the second entry in the minimum is tight for this decomposition, then all of the following equalities hold:
\[
        \alpha(G)=\alpha(G[A]),
        \qquad
        2\tau(G[A])=e(A)+\mu(G[A]),
\]
\[
        e(A)=e(B),
        \qquad
        \mu(G[A])+\mu(G[B])=\mu(G),
\]
and
\[
        e(B)-\mu(G[B])=\binom c2-\floor{c/2}.
\]
Together with Proposition~\ref{prop:general_location}, this reduces an isomorphism-level equality classification to the defect-zero and defect-one decompositions for which the proof inequalities are simultaneously sharp.
\end{proposition}

\begin{proof}
The first assertion follows immediately from the proof of Theorem \ref{thm:general_exact}.  If the second entry is tight, equality must hold in each inequality in the chain
\[
        2x\leq e(A)+\mu(G[A])\leq e(B)+\mu(G)-\mu(G[B])
        \leq \mu(G)+\binom c2-\floor{c/2}.
\]
Unwinding the equalities gives the listed conditions.
\end{proof}

\begin{theorem}[Matching-cover defect criterion for equality in the exact general bound]\label{thm:general_equality_classification}
Let $G$ be a finite simple graph and put $\mu=\mu(G)$.  If $\mu=0$, then equality in Theorem~\ref{thm:general_exact} holds exactly for edgeless graphs.  If $\mu=1$, then every graph with matching number one attains equality.  If $\mu=2$, then equality holds exactly when $a(G)=\alpha(G)+1$.

Assume now that $\mu\geq3$, and put
\[
        r=\ceil{\sqrt{6\mu}}.
\]
Then $G$ attains equality in Theorem~\ref{thm:general_exact} if and only if one of the following two mutually exclusive alternatives holds:
\[
        a(G)=n(G)-r+1
        \quad\text{and}\quad
        \tau(G)=2\mu,
\]
or
\[
        a(G)=n(G)-r+2
        \quad\text{and}\quad
        \tau(G)=2\mu-1.
\]
Equivalently, if $\langle A,B\rangle$ is an annihilating decomposition, then equality holds if and only if either
\[
        |B|=r-1
        \quad\text{and}\quad
        \tau(G)=2\mu,
\]
or
\[
        |B|=r-2
        \quad\text{and}\quad
        \tau(G)=2\mu-1.
\]
In the second alternative, the second proof inequality in Theorem~\ref{thm:general_exact} is tight, and the equalities listed in Proposition~\ref{prop:general_certificate} hold for the corresponding annihilating decomposition.
\end{theorem}

\begin{proof}
The cases $\mu=0$ and $\mu=1$ follow directly from the formula $\Gamma(0)=0$ and $\Gamma(1)=0$, together with $\alpha(G)\leq a(G)$.  For $\mu=2$, the bound gives $a(G)-\alpha(G)\leq1$, so equality is exactly the condition $a(G)=\alpha(G)+1$.

Now assume $\mu\geq3$.  Let $\langle A,B\rangle$ be an annihilating decomposition and put $c=|B|=n(G)-a(G)$.  If equality holds, Proposition~\ref{prop:general_location} gives
\[
        c\in\{r-2,r-1\}
\]
and
\[
        \tau(G)=a(G)-\alpha(G)+c\in\{2\mu-1,2\mu\}.
\]
More precisely, $c=r-1$ gives $\tau(G)=2\mu$, while $c=r-2$ gives $\tau(G)=2\mu-1$.  Since $c=n(G)-a(G)$, these are exactly the two alternatives displayed in the theorem.

Conversely, suppose first that $a(G)=n(G)-r+1$ and $\tau(G)=2\mu$.  Then
\[
        a(G)-\alpha(G)=a(G)-n(G)+\tau(G)=2\mu-r+1=\Gamma(\mu).
\]
Similarly, if $a(G)=n(G)-r+2$ and $\tau(G)=2\mu-1$, then
\[
        a(G)-\alpha(G)=a(G)-n(G)+\tau(G)=2\mu-r+1=\Gamma(\mu).
\]
Thus equality holds in Theorem~\ref{thm:general_exact}.  The final statement is precisely the last assertion of Proposition~\ref{prop:general_location} together with Proposition~\ref{prop:general_certificate}.
\end{proof}

\begin{remark}
The extremal construction in Theorem \ref{thm:general_exact} explains why the arbitrary-graph problem differs sharply from the K\"onig--Egerv\'ary problem.  Low-degree odd components, especially triangles, can contribute two units to the vertex-cover side of the gap while using only one matching edge; a dense clique supplies the edge budget needed for the low-degree part to be annihilating.
\end{remark}

\section{A TxGraffiti inequality as a machine-conjecture case study}\label{section:txgraffiti}

The Havel--Hakimi residue is one of the classical invariants arising from the Graffiti tradition.  We use the following standard comparison theorem.

\begin{theorem}[Caro--Wei; Favaron--Mah\'eo--Sacl\'e; Griggs--Kleitman]\label{thm:res_le_alpha}
For every graph $G$,
\[
        W(G):=\sum_{v\in V(G)}\frac{1}{d(v)+1}\leq \res(G)\leq \alpha(G).
\]
\end{theorem}

\noindent
The lower bound $W(G)\leq\alpha(G)$ is the classical Caro--Wei theorem; see Caro \cite{Caro1979} and Wei \cite{Wei1981}.  The stronger comparison $W(G)\leq\res(G)$ and the residue bound $\res(G)\leq\alpha(G)$ are due to Favaron, Mah\'eo and Sacl\'e \cite{FavaronMaheoSacle1991}; Griggs and Kleitman later gave a short proof of the residue lower bound for independence \cite{GriggsKleitman1994}.  Informally, the Havel--Hakimi reductions can be followed inductively so that the zero entries remaining at termination certify an independent set of at least that size.  In this paper we use only the displayed standard consequences.

TxGraffiti operates with finite versioned snapshot tables of graph invariants and Boolean predicates.  It returns inequalities that are true on the snapshot, but these outputs are conjectures until proved outside the table.  In the terminology of \cite{Davila2026}, the following statement was produced after adding the derived target $\alpha(G)\Delta(G)$, the product of the independence number and maximum degree, and fitting a multivariate lower-bound template involving $a(G)$ and $\res(G)$.

\begin{conjecture}[TxGraffiti]\label{conj:txgraffiti}
If $G$ is a connected graph with $n(G)\geq 3$, then
\[
        \alpha(G)\geq \frac{a(G)+\res(G)}{\Delta(G)}.
\]
\end{conjecture}

The following table summarizes the proof status of the machine output.  TxGraffiti supplies the first line as a table-true conjecture on its snapshot; the remaining lines are the structural ingredients and sharpness checks supplied here.
\begin{center}
\small
\begin{tabular}{p{0.28\linewidth}|p{0.62\linewidth}}
\hline
TxGraffiti output & $\displaystyle \alpha(G)\geq \frac{a(G)+\res(G)}{\Delta(G)}$ \\ \hline
Hypotheses & $G$ connected and $n(G)\geq3$ \\ \hline
Main proof ingredients & $a(G)\leq(\Delta(G)-1)\alpha(G)$ and $\res(G)\leq\alpha(G)$ \\ \hline
Why $n\geq3$ is needed & $K_2$ violates the inequality \\ \hline
Why connectedness is needed & $C_3\cup K_2$ violates the inequality \\ \hline
Equality witnesses & $P_3$, $C_3$, $C_4$, and $K_4$ \\ \hline
\end{tabular}
\end{center}

The conjecture is not a formal consequence of the two basic inequalities $\alpha(G)\leq a(G)$ and $\res(G)\leq\alpha(G)$.  The missing ingredient is an upper bound on the annihilation number itself in terms of $\Delta(G)$ and $\alpha(G)$.  Gupta's recent proof supplies such a bound through the Caro--Wei sum $W(G)$, namely
\[
        a(G)\leq \frac{\Delta(G)+1}{2}W(G),
\]
and combines it with the classical bounds recalled in Theorem~\ref{thm:res_le_alpha} \cite{Gupta2026}.  The proof below follows a different route: it first proves the structural estimate $a(G)\leq(\Delta(G)-1)\alpha(G)$ by annihilating decompositions, Brooks' theorem, and the forest bound.

Gupta's estimate is also useful inside the present paper.  The proof of his Caro--Wei bound contains more information than the displayed inequality alone.  Keeping the nonnegative slack terms gives an exact refinement, which can be combined with our matching-number theorem to produce a combined bracket for the independence number.

For a graph with at least one edge, write the degree sequence as
\[
        d_1\leq d_2\leq\cdots\leq d_n,
\]
put $a=a(G)$ and $\Delta=\Delta(G)$, and define the Caro--Wei sum
\[
        W(G)=\sum_{i=1}^n\frac{1}{d_i+1}.
\]
Let
\[
        S_H(G)=\sum_{i=1}^{a}d_i
\]
be the degree sum of the annihilation head.  Define the Gupta slack
\[
\begin{aligned}
        \sigma_\Delta(G)=&\frac{\Delta(n-a)-S_H(G)}{2\Delta} \\
        &+\sum_{i=1}^{a}\frac{(\Delta-d_i)(\Delta-d_i-1)}{2\Delta(d_i+1)}
        +\sum_{i=a+1}^{n}\frac{\Delta-d_i}{2(d_i+1)} .
\end{aligned}
\]
Each summand is nonnegative.

\begin{theorem}[Refined Gupta slack identity]\label{thm:refined_gupta}
Every graph $G$ with at least one edge satisfies
\[
        a(G)+\sigma_\Delta(G)=\frac{\Delta(G)+1}{2}W(G).
\]
In particular,
\[
        a(G)\leq \frac{\Delta(G)+1}{2}W(G),
\]
with the loss from equality measured exactly by $\sigma_\Delta(G)$.
\end{theorem}

\begin{proof}
Let $H=\{1,\ldots,a\}$ and $T=\{a+1,\ldots,n\}$ be the annihilation head and tail in the nondecreasing degree sequence.  Since $a$ is the annihilation number,
\[
        S_H(G)=\sum_{i\in H}d_i\leq m(G).
\]
Hence
\[
        \sum_{i\in T}d_i=2m(G)-S_H(G)\geq S_H(G),
\]
and, because each tail degree is at most $\Delta$,
\[
        S_H(G)\leq \Delta(n-a).
\]
For every integer $0\leq k\leq\Delta$, the identity
\[
        \frac{1}{k+1}
        =\frac{2}{\Delta+1}-\frac{k}{\Delta(\Delta+1)}
        +\frac{(\Delta-k)(\Delta-k-1)}{\Delta(\Delta+1)(k+1)}
\]
refines the pointwise inequality used in Gupta's proof.  For tail degrees we also have the exact identity
\[
        \frac{1}{k+1}=\frac{1}{\Delta+1}+\frac{\Delta-k}{(\Delta+1)(k+1)}.
\]
Summing the first identity over $H$ and the second over $T$ gives
\[
\begin{aligned}
        W(G)=&\frac{2a}{\Delta+1}-\frac{S_H(G)}{\Delta(\Delta+1)}
        +\sum_{i=1}^{a}\frac{(\Delta-d_i)(\Delta-d_i-1)}{\Delta(\Delta+1)(d_i+1)}\\
        &+\frac{n-a}{\Delta+1}
        +\sum_{i=a+1}^{n}\frac{\Delta-d_i}{(\Delta+1)(d_i+1)} .
\end{aligned}
\]
Multiplying by $(\Delta+1)/2$ and rearranging gives exactly
\[
        \frac{\Delta+1}{2}W(G)=a(G)+\sigma_\Delta(G).
\]
The inequality follows because $\sigma_\Delta(G)\geq0$.
\end{proof}

\begin{corollary}[Equality in Gupta's Caro--Wei annihilation bound]\label{cor:gupta_equality}
Let $G$ be a graph with at least one edge, degree sequence $d_1\leq\cdots\leq d_n$, maximum degree $\Delta$, annihilation number $a=a(G)$, and annihilation-head degree sum
\[
        S_H(G)=\sum_{i=1}^{a}d_i.
\]
Then equality holds in Gupta's inequality
\[
        a(G)=\frac{\Delta(G)+1}{2}W(G)
\]
if and only if all three of the following degree-sequence conditions hold:
\begin{enumerate}
\item[(i)] $S_H(G)=\Delta(n-a)$;
\item[(ii)] every head degree satisfies $d_i\in\{\Delta-1,\Delta\}$ for $1\leq i\leq a$;
\item[(iii)] every tail degree satisfies $d_i=\Delta$ for $a+1\leq i\leq n$.
\end{enumerate}
Equivalently, equality holds if and only if $\sigma_\Delta(G)=0$.
\end{corollary}

\begin{proof}
By Theorem~\ref{thm:refined_gupta}, equality in Gupta's bound is equivalent to $\sigma_\Delta(G)=0$.  Since $\sigma_\Delta(G)$ is a sum of nonnegative terms, it vanishes precisely when each term vanishes.  The first slack term gives $S_H(G)=\Delta(n-a)$.  The head summand
\[
        \frac{(\Delta-d_i)(\Delta-d_i-1)}{2\Delta(d_i+1)}
\]
vanishes exactly when $\Delta-d_i\in\{0,1\}$, that is, $d_i\in\{\Delta,\Delta-1\}$.  The tail summand
\[
        \frac{\Delta-d_i}{2(d_i+1)}
\]
vanishes exactly when $d_i=\Delta$.  These are the three displayed conditions.
\end{proof}

\begin{corollary}[Combined computable independence bracket]\label{cor:combined_bracket}
Let $G$ be a graph with at least one edge.  Then
\[
        \max\left\{\res(G),\, W(G),\, a(G)-\Gamma(\mu(G))\right\}
        \leq \alpha(G)\leq a(G).
\]
Moreover,
\[
        a(G)=\frac{\Delta(G)+1}{2}W(G)-\sigma_\Delta(G)
        \leq \frac{\Delta(G)+1}{2}\res(G)-\sigma_\Delta(G).
\]
\end{corollary}

\begin{proof}
Theorem~\ref{thm:res_le_alpha} gives $W(G)\leq\res(G)\leq\alpha(G)$.  Theorem~\ref{thm:general_exact} gives $a(G)-\alpha(G)\leq\Gamma(\mu(G))$, and hence $a(G)-\Gamma(\mu(G))\leq\alpha(G)$.  Pepper's annihilation bound gives $\alpha(G)\leq a(G)$.  Theorem~\ref{thm:refined_gupta} gives the displayed identity for $a(G)$, and Theorem~\ref{thm:res_le_alpha} gives $W(G)\leq\res(G)$.
\end{proof}

\begin{corollary}[Gupta--residue annihilation-gap bound]\label{cor:gupta_residue_gap}
Let $G$ be a graph with at least one edge.  Then
\[
        a(G)-\alpha(G)
        \leq
        \min\left\{
        \Gamma(\mu(G)),\,
        a(G)-\res(G)
        \right\}.
\]
Moreover, writing $\Delta=\Delta(G)$, $a=a(G)$, $R=\res(G)$, $W=W(G)$, and $\sigma=\sigma_\Delta(G)$, one has the exact identity
\[
        a(G)-\res(G)
        =
        \frac{(\Delta-1)a(G)-2\sigma_\Delta(G)}{\Delta+1}
        -\bigl(\res(G)-W(G)\bigr).
\]
Denote this last expression by $\Theta_{\rm GR}(G)$.  Consequently,
\[
        a(G)-\alpha(G)
        \leq
        \min\{\Gamma(\mu(G)),\Theta_{\rm GR}(G)\}.
\]
In particular, any graph attaining equality in the exact matching bound of Theorem~\ref{thm:general_exact} must satisfy
\[
        (\Delta(G)+1)\Gamma(\mu(G))
        \leq
        (\Delta(G)-1)a(G)-2\sigma_\Delta(G)
        -(\Delta(G)+1)\bigl(\res(G)-W(G)\bigr).
\]
\end{corollary}

\begin{proof}
The first inequality follows from Theorem~\ref{thm:general_exact} and Theorem~\ref{thm:res_le_alpha}: since $\alpha(G)\geq\res(G)$,
\[
        a(G)-\alpha(G)\leq a(G)-\res(G),
\]
while Theorem~\ref{thm:general_exact} gives $a(G)-\alpha(G)\leq\Gamma(\mu(G))$.
By Theorem~\ref{thm:refined_gupta},
\[
        a(G)=\frac{\Delta+1}{2}W-\sigma.
\]
Therefore
\[
\begin{aligned}
        a(G)-\res(G)
        &=\frac{\Delta+1}{2}W-\sigma-R\\
        &=\frac{(\Delta-1)a(G)-2\sigma}{\Delta+1}-(R-W),
\end{aligned}
\]
where the second equality follows by substituting $W=2(a+\sigma)/(\Delta+1)$.  This proves the displayed Gupta--residue gap bound.  If $a(G)-\alpha(G)=\Gamma(\mu(G))$, then the same bound forces the final displayed necessary condition.
\end{proof}

\begin{corollary}[Quantitative domination of the TxGraffiti lower bound]\label{cor:gupta_dominates_tx}
If $G$ has $\Delta(G)\geq3$, then
\[
        \frac{a(G)+\res(G)}{\Delta(G)}\leq \res(G)\leq\alpha(G).
\]
More precisely, write $\Delta=\Delta(G)$, $R=\res(G)$, $W=W(G)$, and $\sigma=\sigma_\Delta(G)$.  Then
\[
        R-\frac{a(G)+R}{\Delta}
        = \frac{(\Delta-3)R}{2\Delta}
        +\frac{(\Delta+1)(R-W)}{2\Delta}
        +\frac{\sigma}{\Delta} .
\]
Consequently, equality in the domination of the TxGraffiti lower bound by $\res(G)$ occurs if and only if $\Delta(G)=3$, $W(G)=\res(G)$, and $\sigma_\Delta(G)=0$.
\end{corollary}

\begin{proof}
By Theorem~\ref{thm:refined_gupta},
\[
        a(G)=\frac{\Delta+1}{2}W-\sigma.
\]
Therefore
\[
\begin{aligned}
R-\frac{a(G)+R}{\Delta}
&=\frac{(\Delta-1)R-a(G)}{\Delta}\\
&=\frac{(\Delta-1)R-\frac{\Delta+1}{2}W+\sigma}{\Delta}\\
&=\frac{(\Delta-3)R}{2\Delta}
 +\frac{(\Delta+1)(R-W)}{2\Delta}
 +\frac{\sigma}{\Delta}.
\end{aligned}
\]
The right-hand side is nonnegative for $\Delta\geq3$, since $R\geq W$ and $\sigma\geq0$.  This proves the sharpened inequality and the equality condition.
\end{proof}

\begin{corollary}[Equality in the TxGraffiti inequality]\label{cor:tx_equality_gupta}
Let $G$ be connected with $n(G)\geq3$.  If $\Delta(G)=2$, then equality in Theorem~\ref{thm:txgraffiti} holds if and only if $\res(G)=\alpha(G)$.  If $\Delta(G)\geq3$, then equality in Theorem~\ref{thm:txgraffiti} holds if and only if
\[
        \Delta(G)=3,\qquad W(G)=\res(G)=\alpha(G),\qquad \sigma_\Delta(G)=0.
\]
\end{corollary}

\begin{proof}
If $\Delta(G)=2$, then $G$ is a path or a cycle.  In both cases $a(G)=\alpha(G)$, and hence
\[
        \alpha(G)=\frac{a(G)+\res(G)}{2}
\]
holds if and only if $\res(G)=\alpha(G)$.

Now assume $\Delta(G)\geq3$.  By Corollary~\ref{cor:gupta_dominates_tx},
\[
        \frac{a(G)+\res(G)}{\Delta(G)}\leq \res(G)\leq \alpha(G).
\]
Thus equality in Theorem~\ref{thm:txgraffiti} holds if and only if both inequalities in this chain are equalities.  The first equality is characterized by Corollary~\ref{cor:gupta_dominates_tx}: it is equivalent to $\Delta(G)=3$, $W(G)=\res(G)$, and $\sigma_\Delta(G)=0$.  The second equality is $\res(G)=\alpha(G)$.  This gives exactly the displayed conditions.
\end{proof}

\begin{lemma}\label{lem:a_delta_alpha}
If $G$ is a connected graph with $n(G)\geq 3$, then
\[
        a(G)\leq (\Delta(G)-1)\alpha(G).
\]
\end{lemma}

\begin{proof}
Write $n=n(G)$, $m=m(G)$, $a=a(G)$, $\alpha=\alpha(G)$, and $\Delta=\Delta(G)$.  Since $G$ is connected and $n\geq3$, we have $\Delta\geq2$.  Let $\langle A,B\rangle$ be an annihilating decomposition of $G$.  Since $\sum_{v\in A}d(v)\leq m$, we have $\sum_{v\in B}d(v)\geq m$.  Hence
\[
        m\leq \sum_{v\in B}d(v)\leq \Delta |B|=\Delta(n-a),
\]
which gives
\begin{equation}\label{eq:a_nmDelta_case}
        a\leq n-\frac{m}{\Delta}.
\end{equation}

First suppose that $G$ is not a tree.  Then $m\geq n$.  If $G$ is neither a complete graph nor an odd cycle, Brooks' theorem gives $\chi(G)\leq \Delta$ \cite{Brooks1941}.  A coloring with at most $\Delta$ color classes has an independent color class of size at least $n/\Delta$; hence
\[
        \alpha\geq \frac{n}{\Delta}.
\]
Using \eqref{eq:a_nmDelta_case}, we obtain
\[
        a\leq n-\frac{m}{\Delta}
        \leq n-\frac{n}{\Delta}
        =\frac{(\Delta-1)n}{\Delta}
        \leq (\Delta-1)\alpha.
\]
If $G$ is an odd cycle, then $\Delta=2$ and $a=\alpha=(n-1)/2$, so the desired inequality holds.  If $G=K_{\Delta+1}$, then $\alpha=1$ and
\[
        a=\floor{\frac{\Delta+1}{2}}\leq \Delta-1
\]
for $\Delta\geq 2$, again giving $a\leq (\Delta-1)\alpha$.

Now suppose that $G$ is a tree.  If $\Delta=2$, then $G$ is a path and $a=\alpha=\ceil{n/2}$, so the result holds.  If $\Delta\geq 3$, Theorem \ref{thm:tree} and Lemma \ref{lem:KE_mu_le_alpha} give
\[
        a\leq \alpha+\frac{\mu(G)-1}{2}
        \leq \alpha+\frac{\alpha-1}{2}
        =\frac{3\alpha-1}{2}
        \leq 2\alpha
        \leq (\Delta-1)\alpha.
\]
This completes the proof.
\end{proof}

\begin{theorem}[An annihilating-decomposition proof of the TxGraffiti inequality]\label{thm:txgraffiti}
If $G$ is a connected graph with $n(G)\geq 3$, then
\[
        \alpha(G)\geq \frac{a(G)+\res(G)}{\Delta(G)}.
\]
\end{theorem}

\begin{proof}
By Lemma \ref{lem:a_delta_alpha},
\[
        a(G)\leq (\Delta(G)-1)\alpha(G).
\]
By Theorem \ref{thm:res_le_alpha},
\[
        \res(G)\leq \alpha(G).
\]
Adding the two inequalities gives
\[
        a(G)+\res(G)
        \leq (\Delta(G)-1)\alpha(G)+\alpha(G)
        =\Delta(G)\alpha(G).
\]
Since $G$ is connected and $n(G)\geq 3$, we have $\Delta(G)\geq 2$.  Dividing by $\Delta(G)$ proves the theorem.
\end{proof}

\begin{proposition}[Necessity of the TxGraffiti hypotheses]\label{prop:tx_hypotheses_sharp}
Both hypotheses in Theorem \ref{thm:txgraffiti} are necessary.
\end{proposition}

\begin{proof}
The order assumption cannot be removed.  For $K_2$,
\[
        \alpha(K_2)=1,
        \qquad
        a(K_2)=1,
        \qquad
        \res(K_2)=1,
        \qquad
        \Delta(K_2)=1.
\]
Thus
\[
        \frac{a(K_2)+\res(K_2)}{\Delta(K_2)}=2>1=\alpha(K_2).
\]

Connectedness cannot be removed either.  Let $G=C_3\cup K_2$.  Then
\[
        \alpha(G)=2,
        \qquad
        a(G)=3,
        \qquad
        \res(G)=2,
        \qquad
        \Delta(G)=2.
\]
Therefore,
\[
        \frac{a(G)+\res(G)}{\Delta(G)}=\frac52>2=\alpha(G).
\]
This proves both claims.
\end{proof}

\begin{proposition}[Small sharp witnesses]\label{prop:tx_small_sharp}
The inequality in Theorem \ref{thm:txgraffiti} is attained by several connected graphs, including $P_3$, $C_3$, $C_4$, and $K_4$.
\end{proposition}

\begin{proof}
The required values are obtained directly from the definitions and from the Havel--Hakimi reduction.  For example,
\[
\begin{array}{c|cccc}
G&\alpha(G)&a(G)&\res(G)&\Delta(G)\\ \hline
P_3&2&2&2&2\\
C_3&1&1&1&2\\
C_4&2&2&2&2\\
K_4&1&2&1&3
\end{array}
\]
In each row, $\Delta(G)\alpha(G)=a(G)+\res(G)$.
\end{proof}

\begin{remark}
This section illustrates the intended role of automated conjecturing.  TxGraffiti supplied a compact, human-readable, table-true inequality involving three independently studied invariants.  The proof above uses neither the finite table nor optimization; the table output instead identifies the correct combination of invariants, while the theorem follows from annihilating decompositions, Brooks' theorem, the forest bound, and the residue inequality.
\end{remark}

\subsection*{Comparison with Gupta's Caro--Wei proof}
Gupta's proof of the same TxGraffiti inequality is stronger at the degree-sequence level: it proves
\[
        a(G)\leq \frac{\Delta(G)+1}{2}W(G),
\]
where $W(G)=\sum_{v\in V(G)}1/(d(v)+1)$ is the Caro--Wei sum \cite{Gupta2026}.  Together with the standard bounds recalled in Theorem~\ref{thm:res_le_alpha}, this immediately yields the TxGraffiti inequality for connected graphs of maximum degree at least three; maximum degree two is then handled directly.  Corollary~\ref{cor:gupta_dominates_tx} records a useful consequence: for $\Delta\geq3$, the TxGraffiti lower bound $(a+\res)/\Delta$ is dominated by the older residue bound $\res$.

Our proof is complementary.  It does not use the Caro--Wei sum; instead, it derives the coarser but structural estimate
\[
        a(G)\leq(\Delta(G)-1)\alpha(G)
\]
from annihilating decompositions and the matching-number theory developed above.  Theorem~\ref{thm:refined_gupta} extracts the slack hidden in Gupta's argument and turns it into the exact identity
\[
        a(G)+\sigma_\Delta(G)=\frac{\Delta(G)+1}{2}W(G).
\]
Combining this identity with our exact matching theorem gives the combined bracket in Corollary~\ref{cor:combined_bracket} and the Gupta--residue annihilation-gap bound in Corollary~\ref{cor:gupta_residue_gap}, while Corollary~\ref{cor:gupta_dominates_tx} quantifies exactly how far the TxGraffiti lower bound lies below the residue bound when $\Delta(G)\geq3$.

\section{Concluding remarks}\label{section:conclusion}
We have given a sharp matching-number theory for the annihilation gap $a(G)-\alpha(G)$ in several natural graph classes.  The arbitrary-graph theorem gives a closed exact bound in terms of $\mu(G)$, while the forest, bipartite, and K\"onig--Egerv\'ary theorems explain how the extremal value changes under classical structure.  The TxGraffiti section then shows that these annihilation methods give an independent structural proof of a machine-generated inequality involving the annihilation number, the Havel--Hakimi residue, and maximum degree, complementing the recent Caro--Wei proof of Gupta.  Combining the two approaches yields the bracket in Corollary~\ref{cor:combined_bracket} and the Gupta--residue gap bound in Corollary~\ref{cor:gupta_residue_gap}, while Corollary~\ref{cor:gupta_equality} identifies equality in Gupta's Caro--Wei annihilation bound by explicit degree-sequence conditions and Corollary~\ref{cor:tx_equality_gupta} characterizes equality in the TxGraffiti inequality in terms of the same slack.  Thus the slack term $\sigma_\Delta(G)$ measures exactly the loss in Gupta's inequality and relates the computable quantities $\res(G)$, $W(G)$, $a(G)$, $\mu(G)$, and $\Delta(G)$ to $\alpha(G)$.

The equality cases of the two principal estimates are now also isolated at the level needed for the present paper.  Theorem~\ref{thm:general_equality_classification} gives a matching-cover defect criterion for equality in the exact arbitrary-graph matching bound: for $\mu\geq3$, equality is equivalent to one of two parameter alternatives with defect zero or defect one.  Theorem~\ref{thm:bipartite_connected_classification} gives a structural template for the connected bipartite real-equality case, including the precise role of the cross-edges.  A full isomorphism-level enumeration of the possible defect-zero/defect-one decompositions and of the possible bipartite cross-edge patterns is a natural sequel problem, but it is no longer needed for the sharp inequalities themselves.  Finally, forbidding cliques leads to a different fixed-matching extremal theory; that direction is pursued in a separate sequel focused on the exact $K_4$-free function, matched cores, and bounded blow-up type graphs.

The sequel also gives useful context for the scope of the present methods.  Once a maximum matching is fixed, its $K_4$-free candidates are encoded by finite matched cores and bounded blow-ups of type graphs.  A fixed-defect form of the Tutte--Berge maximum-matching formula \cite{Tutte1947,Berge1958}, together with a Gallai-defect viewpoint \cite{Gallai1959}, organizes the sequel's frontier obstructions.  These observations are not needed for the proofs above, but they explain why the unrestricted extremal construction in Theorem~\ref{thm:general_exact} is clique-driven, while forbidding $K_4$ leads to a different reservoir-and-blow-up theory.  In that sequel, the Turan--Zykov viewpoint and bounded-matching Turan input \cite{AlonFrankl2024,Mantel1907,Turan1941,Zykov1949} enter only after the matching-core structure has been isolated.

\appendix
\section{Computational reproducibility notes}\label{appendix:computations}
The computations in this appendix are not used in the proofs.  They are included to make the boundary cases reproducible.  For $1\leq \mu\leq 17$, the exact arbitrary-graph value
\[
        \Gamma(\mu)=2\mu+1-\ceil{\sqrt{6\mu}}
\]
is
\[
\begin{array}{c|ccccccccccccccccc}
\mu&1&2&3&4&5&6&7&8&9&10&11&12&13&14&15&16&17\\ \hline
\Gamma(\mu)&0&1&2&4&5&7&8&10&11&13&14&16&18&19&21&23&24
\end{array}
\]
The following minimal script computes $a(G)$ from the degree sequence and verifies the formula against any supplied finite graph list.
\begin{verbatim}
def annihilation_number(G):
    deg = sorted(dict(G.degree()).values())
    m = G.number_of_edges()
    total = ans = 0
    for d in deg:
        if total + d <= m:
            total += d
            ans += 1
        else:
            break
    return ans
\end{verbatim}

\section*{Declarations}
The authors declare that they have no conflict of interest.

\end{document}